\documentclass[11pt, letterpaper, oneside, reqno]{amsart}

\usepackage{latexsym, amsmath, amssymb, amsfonts, amscd, amsthm,bm}
\usepackage{amsthm}
\usepackage{t1enc}
\usepackage[mathscr]{eucal}
\usepackage{indentfirst}
\usepackage{graphicx, pb-diagram}
\usepackage{fancyhdr}
\usepackage{fancybox}
\usepackage{enumerate}
\usepackage{xcolor}
\usepackage[all]{xy}
\usepackage{hyperref}
\usepackage{xspace}
\usepackage{tikz}
\usepackage{tikz-cd}
\usetikzlibrary{matrix}
\usepackage{mathrsfs}
\usepackage{relsize}
\usepackage{url}

\usepackage[latin1]{inputenc}
\usepackage[T1]{fontenc}
\usepackage{amstext}
\usepackage[normalem]{ulem}

\theoremstyle{plain}
\newtheorem{thm}{Theorem}[section]
\newtheorem*{thm*}{Theorem} 

\newtheorem{prop}[thm]{Proposition}

\newtheorem{lemma}[thm]{Lemma}

\theoremstyle{definition}
\newtheorem{defi}[thm]{Definition}

\theoremstyle{remark}
\newtheorem{remark}[thm]{Remark}
\newtheorem{ex}[thm]{Example}

\newcommand{\source}{\ensuremath{\mathbf{s}}}
\newcommand{\target}{\ensuremath{\mathbf{t}}}

\newcommand{\RR}{\ensuremath{\mathbb R}}
\newcommand{\bI}{\ensuremath{\bold{I}}}
\newcommand{\II}{\ensuremath{\mathbb I}}
\newcommand{\g}{\ensuremath{\mathfrak{g}}}

\newcommand{\h}{\ensuremath{\mathfrak{h}}}

\newcommand{\id}{\ensuremath{\mathrm{id}}}

\newcommand{\sC}{\mathscr{C}}
\newcommand{\sD}{\mathscr{D}}

\newcommand{\cV}{\mathcal{V}}

\newcommand{\Lie}{\mathcal{L}}
\renewcommand{\d}{\mathrm{d}}

 \definecolor{darkgreen}{rgb}{0.0, 0.17, 0.1}

\definecolor{forest}{rgb}{0,0.5,0}

\newcommand{\ddt}{\frac{\partial}{\partial t}}

\newcommand{\G}{\mathcal{G}}

\newcommand{\tto}{\rightrightarrows}

\def\Lie{{\bf\mathcal{L}\mkern-1,5mu\textit{\textbf{ie}}}}

\begin{document}

\title[Lie groupoids and their natural transformations]{Lie groupoids and their natural transformations}
  
\author{Olivier Brahic}
\email{olivier@ufpr.br}  
\address{Departamento de Matem\'atica - UFPR
Centro Polit\'ecnico - Jardim das Am\'ericas
CP 19081
CEP:81531-980,  Curitiba - Paran\'a, Brazil}

\author{Dion Pasievitch }
\email{dion.rss@gmail.com}
\address{
Universidade Estadual do Paraná - UNESPAR, Campus União da Vitória. 
R. Cel. Amazonas - Centro - CEP:84600-000, União da Vitória - Paran\'a, Brazil}

\date{\today}
\subjclass[2010]{}


\maketitle
\begin{abstract}
 We discuss natural transformations in the context of Lie groupoids, and their infinitesimal counterpart. Our main result is an integration procedure that provides smooth natural transformations between Lie groupoid morphisms. 
\end{abstract}

\setcounter{tocdepth}{1} 
\tableofcontents

\section*{Introduction}

A Lie groupoid is a groupoid together with a compatible smooth structure. The precise definition is obtained by considering groupoids that are \emph{internal} to the category of manifolds: this means that one imposes on both the space of objects and of morphisms of a Lie groupoid to be smooth manifolds, and on all structure maps (source, target, multiplication...) to be smooth. Similarly, a Lie groupoid morphism is just a functor that is smooth in the sense that all its underlying maps are required to be smooth.

From there on, the notion of natural transformation between Lie groupoid morphisms should be clear: it is a natural transformation that is smooth as a map from the space of objects to the space of morphisms. In fact, despite the fact that Lie groupoids have recently been the subject of intensive studies, the notion of natural transformation in this context is hardly addressed, which partly motivated this work.

To any Lie groupoid, there is an associated infinitesimal structure called Lie algebroid. The procedure of associating to a Lie groupoid $\G$ its Lie algebroid $A=\Lie(\G)$ is very similar to that of associating to a Lie group its Lie algebra, and should be thought of as a differentiation process. In particular, it is functorial, namely to any Lie groupoid morphism $F:\G_1\to \G_2$ one can associate a Lie algebroid morphism $\Lie(F):\Lie(\G_1)\to \Lie(\G_2)$.

The inverse procedure, that of associating to a Lie algebroid $A$ a Lie groupoid $\G(A)$, is called \emph{integration} and is more subtle. Indeed, not any Lie algebroid comes from a Lie groupoid.  There are obstructions to the integrability of a Lie algebroid, that were exhibited by Crainic and Fernandes
in \cite{CrFe2}. Note that despite integrability issues, the integration procedure always produces a topological groupoid and, from this point of view, it is functorial. 

Regarding natural transformations,  there are two basic questions that naturally arise in the context of Lie groupoids:
\begin{enumerate}[a)]
\item What is the infinitesimal counterpart of a (smooth) natural transformation for Lie groupoids, if any ?
\item What would be the corresponding integration procedure ?
\end{enumerate}
The aim of this work is to address both questions, to which, surprisingly, we could not find any explicit answer in the literature. In fact, we have no doubt most results in this work are known to some experts  (in one form or another, see \emph{e.g.} \cite{Sul, B, BojKotovStrobl, StashefShreiber, RBrown}) especially for people familiar with classifying spaces.  Nonetheless, after discussing with colleagues, we received enough encouragements to convince ourselves that the gap was worth filling.

 Our main result is the aforementioned integration procedure, that states the following.

\begin{thm*}
Let $A_M$ and $A_N$ be Lie algebroids over smooth manifolds $M$ and $N$ respectively, together with a smooth family $\Phi_t:A_M\to A_N$ of Lie algebroid morphisms parametrized by $t\in [0,1]$. 

We denote by $\phi_t: M\to N$ the  underlying family of smooth maps, and by 
$$\Phi_t^*:\bigl(\Omega^\bullet(A_N), d_{A_N}\bigr)\longrightarrow  \bigl(\Omega^\bullet(A_M), d_{A_M}\bigr)$$ the smooth family of morphisms of cochain complexes induced by $\Phi_t$.

Assume that there exists a smooth family $\theta_t\in \Gamma(\phi_t^*A_N)$ of sections of $A_N$ supported by $\phi_t$, that satisfies the following condition:
\begin{equation}
 \ddt\Phi_t^*= d_{A_M}\circ i_{\theta_t}-i_{\theta_t}\circ d_{A_N},
\end{equation}
where $i_{\theta_t}:\Omega^\bullet(A_N)\to \Omega^{\bullet-1}(A_M)$ is the 
 obvious family of maps obtained  by contraction with $\theta_t$.
 
Then, the following assertions hold:
\begin{enumerate}[i)]
\item The morphisms $\Phi_0^*$ and $\Phi_1^*$ are chain homotopic.
\item Whenever $A_M$ and $A_N$ are integrable Lie algebroids, the assignment:
$$\eta(m):=[\theta(m)dt]_{A_M}$$
defines a smooth natural transformation $\eta: F_0\Rightarrow F_1$, where
$F_0,F_1: \G(A_M)\to \G(A_N)$
denote the Lie groupoid morphisms integrating $\Phi_0$ and $\Phi_1$ respectively.

Here, we denoted by $[\theta(m)dt]_{A_N}$ the $A_N$-homotopy class of the $A_N$-path $\theta(m)dt:TI\to A_N$.
\end{enumerate}
\end{thm*}

Here, note that the item $i)$ is essentially due to Balcerzak \cite{B}. However, as we shall see, it is  meaningful to have both natural transformations and chain homotopies in a same statement, as these correspond to a same geometric entity. 

The paper is organized as follows:
\begin{itemize}
\item  We first briefly review, in \S \ref{sec:generalities}  basic facts about Lie groupoids and  Lie algebroids, their morphisms and their integration, establishing a few notations on the way. Then we review in \S \ref{sec:naturaltandhomotopies} the notion of natural transformation, explain how they can be seen as \emph{discrete} homotopies, and how this translates in the smooth setting. Although elementary, this part is crucial in order to answer to the question $a)$.

\item  As it turns out, there is no real infinitesimal counterpart to a natural transformation, a fact which we  clarify in the Remark \ref{rem:infinitesimalnaturaltransf}. This forces us to rather consider what we call \emph{natural homotopies}, and can be seen as smooth families of natural transformations.
 
\item   In the subsequent Section \S\ref{sec:infinitesimal}, we look at the infinitesimal counterpart of a natural homotopy, and how this expresses in terms of DGAs. The section \S \ref{sec:integration} is then devoted to the main Theorem, namely an integration procedure leading to a smooth natural transformation. Finally, we briefly look at examples, first  in the context of fundamental groupoids, and then look at deformations retract for Lie algebroids.
\end{itemize} 

Regarding the existing literature, let us emphasize the following interesting points. Before the notion of Lie algebroid even emerged, the relevance of the notion of homotopy between maps of DGAs was already pointed out by Sullivan in \cite{Sul}. 

In the context of Lie algebroids, homotopies between Lie algebroid morphisms appear in the earlier work of Kubarski \cite{Kub} on Chern-Weil homomorphisms (see also \cite{BKW, B} and \cite{CrFe5} where homotopy invariance is fundamental)  then studied later on, in connection with Poisson sigma models, by Kotov-Strobl  \cite{BojKotovStrobl}. Another closely related notion include, to a certain extent, the work of Abad and Crainic on representations up to homotopy \cite{CrAb} (see also \cite{BO}). However, the interpretation of such homotopies in terms of natural transformation never clearly appears in any of these works.  

A generalization to $L_\infty$-algebras with a strong categorical point of view can be found in \cite{StashefShreiber}. The application of these results to Lie algebroids however, is not immediate, unless considering a Lie algebroid as an infinite dimensional Lie algebra which turns out quite restrictive, as we point out in the Remark \ref{rem:infinitesimalnaturaltransf}). Note however that we laid out our construction in such a way that it should be easy to carry through in the graded setting (namely for $NQ$-manifolds).

Finally, we would like to point out the  work of H. Colman \cite{HC1, HC2} that generalizes the notion of homotopy between Lie groupoid morphisms to generalized morphisms. In particular, although the notion of Morita equivalence for Lie algebroid is a subtle one, an infinitesimal counterpart of the construction of H. Colman seems like an interesting venue for future work.

\section{Generalities on Lie groupoids and Lie algebroids}\label{sec:generalities}

\subsection{Lie Groupoids}
 A Lie groupoid is, roughly speaking, a groupoid object in the category of smooth manifolds. Indeed, we need to add some extra requirement to make up for the fact the category of smooth manifolds does not have pullbacks. More precisely, a \emph{Lie groupoid} over a manifold $M$ is a small groupoid $\G_M\rightrightarrows M$ such that both $\G_M$ and $M$ are smooth manifolds, all the structure maps are smooth and the source map (or the target) is a surjective submersion. The last requirement is to  ensure the space of composable arrows $\mathcal{G}_M\times_M \mathcal{G}_M$ is a smooth manifold, so that it makes sense to require the multiplication map to be smooth. 

A \emph{Lie groupoid morphism} from a Lie groupoid $\G_M\rightrightarrows M$ to another $\G_N\rightrightarrows N$ is simply a functor $F:\mathcal{G}_M\longrightarrow \mathcal{G}_N$ which is a smooth map both at the level of objects and of morphisms. The corresponding category will be denoted by $\mathsf{LieGpd}$. General references for Lie groupoids are \cite{MK2} and \cite{MMrcunbook}. 

\subsection{Lie algebroids}

The infinitesimal counterpart of Lie groupoids are Lie algebroids. A \emph{Lie algebroid} over a manifold $M$ is a vector bundle $A\longrightarrow M$ together with a Lie bracket $[\cdot, \cdot]_{A}$ defined on its space of sections $\Gamma(A)$ and a vector bundle morphism $\sharp_{A}:A\longrightarrow TM$ covering the identity $\mathrm{id}_M$, called the \emph{anchor}, and satisfying the \emph{Leibniz rule}
\begin{align*}
[a, f\cdot b]_{A}=f\cdot [a, b]_A+\mathcal{L}_{\sharp_{A} (a)}(f)\cdot b
\end{align*}
for every $a, b\in \Gamma(A)$ and $f\in C^\infty(M)$, where $\mathcal{L}_{\sharp_{A} (a)}$ stands for the Lie derivative along the vector field $\sharp_{A}(a)$.

Any lie algebroid $A\longrightarrow M$ gives rise to a cochain complex  $(\Omega^\bullet(A), d_{A})$ defined in degree $k$ by:
\begin{align*}
\Omega^k(A):=\Gamma(\wedge^k A^*)
\end{align*} 
and whose differential is the degree $1$ linear map $d_{A}:\Omega^\bullet(A)\longrightarrow \Omega^{\bullet+1}(A)$ given by
\begin{align*}
(d_{A} \varepsilon)(a_0{}_\wedge \ldots{}_\wedge a_k)&:=\sum_{j=0}^k (-1)^j \mathcal{L}_{\sharp_{A} a_j}(\varepsilon(a_0{}_\wedge \ldots{}_\wedge\widehat{a_j}{}_\wedge \ldots{}_\wedge a_k))\\
&+ \sum_{0\leq i<j\leq k} (-1)^{i+j} \varepsilon([a_i, a_j]_{A}{}_\wedge a_0{}_\wedge \ldots{}_\wedge\widehat{a_i}{}_\wedge \ldots \widehat{a_j}{}_\wedge \ldots, a_k),
\end{align*}

for every $\varepsilon \in \Omega^k(A_M)$ and $a_1, \ldots, a_k\in \Gamma(A_M)$. It is a well know fact that this operator squares to zero,  $d_{A_M}^2=0$, and that it is a derivation of the wedge product $\wedge$, that is:
\begin{align*}
d_{A_M}(\alpha\wedge \beta)=d_{A_M}\alpha\wedge \beta+(-1)^{|\alpha|} \alpha\wedge d_{A_M} \beta,
\end{align*}
for every $\alpha, \beta\in \Omega^\bullet(A)$ homogeneous. Therefore, $(\Omega^\bullet(A_M), d_{A_M}, \wedge)$ is a differential graded algebra (DGA for short).

 A \emph{Lie algebroid morphism} from a Lie algebroid $A_M\longrightarrow M$ to another $A_N\longrightarrow N$ is a vector bundle morphism $\Phi:A_M\longrightarrow A_N$ compatible with the anchors and with the brackets $[\cdot, \cdot]_{A_M}$ and $[\cdot, \cdot]_{A_N}$. We refer the reader to \cite{MK2} for the explicit compatibility conditions. Along this work, we shall, however, use a characterization of Lie algebroid morphisms in terms of morphisms of DGAs as we point out in the next section.

\subsection{Lie algebroids as Differential Graded Algebras}\label{sec:LiealgasDGAs}
Recall from \cite{Vain} the fundamental result that a Lie algebroid structure on a vector bundle $A\to M$ is equivalent to a degree $1$ derivation of the graded algebra $(\Omega(A),{}\wedge)$. More precisely, it is equivalent to an operator $d_A:\Omega(A)\to \Omega(A)$ that makes $(\Omega(A),\wedge,d_{A})$ into a  differential graded algebra (DGA, for short) which amounts to requiring that:
$${d}_{A}(\alpha_\wedge\beta)=({d}_{A}\alpha)_\wedge\beta+(-1)^{|\alpha|}\alpha_\wedge({d}_{A}\beta).$$

Further, given Lie algebroids $A_M\to M$ and $A_N\to N$,  a vector bundle map $\Phi:A_M\to A_N$ is a Lie algebroid morphism if and only if the naturally induced map:
$$\Phi^*:\Omega(A_N)\to \Omega(A_M)$$

defined as
\begin{align*}
(\Phi^* \varepsilon)_p(a_1, \ldots, a_k):=\varepsilon_{\phi(p)}(\Phi(a_1), \ldots, \Phi(a_k)),
\end{align*}
for every $\varepsilon\in \Omega^k(A_N)$, $p\in M$ and $a_1, \ldots, a_k\in (A_M)_p$, is a morphism of DGAs, namely \emph{iff} the following conditions hold:
\begin{align*}
\Phi^*(\alpha_\wedge\beta)&=\Phi^*(\alpha)_\wedge \Phi^*(\beta), \\
\Phi^*\circ d_{A_N}&=d_{A_M}\circ \Phi^*.
\end{align*}
Note that any morphism of such DGAs is of the form $\Phi^*$ for some vector bundle map $\Phi:A_M\to A_N$. 

In other words the structure of a  DGA on $\Omega(A)$ entirely captures the  Lie algebroid structure on $A$. In fact, Lie algebroids and their morphisms are algebraically better behaved when described in the above way\footnote{In particular, the notion of morphism between Lie algebroids over different bases is much easier to deal with when working with $\Phi^*$ rather than $\Phi$, see \cite[Prop 2]{BojKotovStrobl}.}

In the sequel, the category of Lie algebroids will be denoted by $\mathsf{LieAlgd}$.

\subsection{Derivations of a Lie algebroid}
The following definitions were borrowed from \cite{BojKotovStrobl}.

\begin{defi}
Given a smooth  map $\phi:M\to N$, and a vector bundle $A_N$ over $N$, a  \emph{section of $A_N$ supported by $\phi$} is a section $\theta$ of the pullback bundle $\theta\in \Gamma(\phi^*A_N)$.
\end{defi}
In particular, if $A_M=TM$ and $A_N=TN$ are tangent bundles, then  $\theta$ is just a \emph{vector field supported by} $\phi$ that is, essentially, a vector tangent at $\phi$ in the space of smooth maps from $M$ to $N$.

\begin{defi}\label{def:phiderivation}
Let  $\Phi:A_M\to A_N$ be a vector bundle map covering $\phi:M\to N$,  and denote by $\Phi^*:\Omega(A_N)\to \Omega(A_M)$ the induced map on forms. 

A \emph{$\Phi^*$-derivation of degree $k$} is a degree $k$ aplication $$\cV: \Omega^\bullet(A_N)\to \Omega^{\bullet+k}(A_M)$$
 such that:
\begin{align}
\cV(\alpha+\beta)&=\cV(\alpha)+\cV(\beta),\\
\label{eq:phiderivationproperty}
\cV(\alpha_\wedge\beta)
&=\cV(\alpha)_\wedge\Phi^*(\beta)+(-1)^{k\cdot |\alpha|}\Phi^* (\alpha)_\wedge\cV(\beta).
\end{align}
In the case $k=0$, $A_M=A_N$, and $\Phi=\id_{A_N}$, we say that $\cV$ is a derivation of $\Omega(A_M)$.
\end{defi}

\begin{prop}\label{prop:derivationandtheta}
Under the assumptions and notations of the Definition \ref{def:phiderivation}, there is a one-to-one correspondence between:
\begin{itemize}
\item $\Phi^*\!$-derivations of degree $-1$,
\item sections $\theta\in\Gamma(\phi^*A_N)$ supported by $\phi$,
\end{itemize} 
which is given as follows: to any section $\theta\in\Gamma(\phi^*A_N)$ supported by $\phi$,  one associates the $\Phi^ *\!$-derivation $\cV:=i_{\theta}^{\Phi}$ defined by:
\begin{equation}\label{eq:leibnizopreatordef}
\bigl\langle i_{\theta}^{\Phi}\beta\, ,\, a_1{}_\wedge \dots{}_\wedge a_{k-1} \bigr\rangle
:=
\bigl\langle \beta\circ \phi\, ,\, \theta{}_\wedge \Phi\circ  a_1{}_\wedge \dots {}_\wedge  \Phi\circ a_{k-1}\bigr\rangle\in C^\infty(M)
\end{equation}
where $\beta\in \Omega^k(A_N)$ and $a_1,\dots ,a_{k-1}\in \Gamma(A_M)$.
\end{prop}

\begin{proof}
First, one easily checks that the equation \eqref{eq:leibnizopreatordef} defines a $\Phi^*\!$-derivation as in Definition \ref{def:phiderivation}. To see that any $\Phi^*\!$-derivation $\cV$ can be obtained in this way, one notices that, since $\Omega(A_N)$ is generated (as a graded algebra) by $\Omega^{\leq 1}(A_N)$, and because of the  $\Phi$-derivation condition \eqref{eq:phiderivationproperty},  $\cV$ is entirely determined by its restriction to ${\Omega^{\leq 1}(A_N)}$. Since $\cV$ has degree $-1$, this restriction vanishes on $C^\infty(N)$ so it reduces to a map:
$$\cV|_{\Omega^1(A_N)}: \Omega^1(A_N)\to C^{\infty}(M)$$
that is $C^{\infty}(M)$-linear by \eqref{eq:phiderivationproperty}  -- here $\Omega^1(A_N)$ is made into a $C^\infty(M)$-module by extension of the scalars via the pullback map $\phi^*:C^\infty(N)\to C^\infty(M)$. In other words, $\cV|_{\Omega^1(A_N)}$ can be seen as a $C^\infty(M)$-linear map $\cV|_{\Omega^1(A_N)}:\Gamma(\phi^*A^*_N)\to C^\infty(M)$ that is, by biduality a section $\theta\in \Gamma(\phi^*A_N)$, which is precisely what the equation \eqref{eq:leibnizopreatordef} expresses for $k=1$. This last argument also makes clear that the assignment  $\theta\mapsto i_\theta^\Phi$ is injective.
\end{proof}

\begin{remark}
The operator $i_\theta^\Phi: \Omega^\bullet(A_N)\to \Omega^{\bullet-1}(A_M)$ defined by \eqref{eq:leibnizopreatordef} is in fact  obtained as a composition: $i_\theta^\Phi=\Phi^\star\circ i_\theta$, where $i_{\theta}$ denotes the interior product with $\theta$, \emph{seen as a map} $i_\theta:\Omega^\bullet(A_N)\to \Omega^{\bullet-1}(\phi^*A_N)$, and  $\Phi^\star:\Omega(\phi^*A_N)\to \Omega(A_M)$ is the obvious map induced by $\Phi$. 

Later on, we will have to deal with computations involving both $i_\theta^\Phi=\Phi^\star\circ i_\theta$ and Lie algebroid differentials. Since the pullback bundle $\phi^*A_N$ does not come with a Lie algebroid structure, and in order to avoid confusions between $\Phi^*$ and $\Phi^\star$ (the latter being irrelevant) we will rather use the notation $i_\theta^\Phi$. 
\end{remark}


\subsection{The Lie functor}
Similarly as Lie groups having an associated Lie algebra, any Lie groupoid $\G_M\tto M$ gives rise to a Lie algebroid $A_M\to M$. Let us briefly describe the procedure, again we refer to \cite{MK2} for more details.  

Given a Lie groupoid $\G\tto M$, one consider the vector bundle $A:=\ker d\mathbf{s}|_{1_M}$, obtained by restricting the vector bundle  $\ker d\mathbf{s}$ to the units in $\G_M$.

The sections of $A_M$ identify with right-invariant vector fields on $\G_M$ (for this to make sense, such vector fields need to be tangent to the $\mathbf{s}$-fibers) thus $\Gamma(A_M)$ comes naturally endowed with a   Lie bracket.

 The anchor map is then obtained as  $\sharp_A:=d\mathbf{t}|_{A}$, and it is easily checked that the Leibniz rule is satisfied. The Lie algebroid $A$ obtained this way will be denoted by $A:=\Lie(\G)$. 

The above procedure turns ou to be  functorial: given a Lie groupoid morphism $F:\G_M\to \G_M$, one obtains a Lie algebroid morphism $\Lie(F):\Lie(\G_M)\to \Lie(\G_N)$ by setting $\Lie(F):=dF|_{A_M}$.

We obtain this way a functor:
$$ \Lie: \mathsf{LieGpd}\to \mathsf{LieAlgd}$$
called the \emph{Lie functor}. Intuitively, this functor captures the infinitesimal phenomenons occurring in the theory of Lie groupoids.

Note that not every Lie algebroid $A$ is of the form $A=\Lie(G)$ for some Lie groupoid $\G$. When it is the case, we say that $A$ is \emph{integrable}, and that $\G$ \emph{integrates} $A$. There are obstructions to the integrability of a Lie algebroid that were unravelled in  \cite{CrFe2}. 

When a Lie algebroid is integrable, there might be several Lie groupoids integrating it. There is however a unique one, provided we require $\mathbf{s}$-simply connectedness: it is usually denoted by $\G(A)$. Furthermore, if $\Phi:A_M\to A_N$ is a morphism between integrable Lie algebroids, there exists a Lie groupoid morphism $F:\G(a_M)\to \G(A_N)$, denoted by $F=\G(F)$, \emph{integrating} $F$, that is, such that $\Lie(F)=\Phi$. Therefore, up to integrability issues, one might think of the \emph{integration} procedure $\G(-)$ as a functor inverse to $\Lie$.  

\section{Natural transformation as discrete homotopies}\label{sec:naturaltandhomotopies}

Consider two abstract (small) groupoids $\sC$ and $\sD$, together with functors $$F_0,F_1:\sC\to \sD,$$
that shall be fixed for the remaining of this section. Since we shall largely discuss the notion of natural transformation, let us recall the  basic definition.

\begin{defi}\label{def:naturaltransf}\label{defnatural} A \emph{natural transformation} $\eta: F_0\Rightarrow F_1$ is given by an assignment that associates to any object $x$ in $\sC$ an  arrow $\eta_x$ in $\sD$, in such a way that that for any arrow $a:x\to y$ in $\sC$, we have $F_1(a)\circ \eta_x=\eta_y\circ F_0(a)$ as illustrated below.
$$ 
\SelectTips{cm}{}
\xymatrix@R=25pt@C=30pt{ 
	\ar@/^1.2pc/@{{}{ }{}}[rr]^{\displaystyle{\sC}}&	x \ar@/_0.5pc/[d]_{a}\ar@/^1,7pc/@{{}{ }{}}[r]^{}="a"
	&                  \ar@/^1,7pc/@{{}{ }{}}[r]^{}="b"
	   \ar@/^1,2pc/@{{}{ }{}}[rr]^{\displaystyle{\sD}} & 	F_0(x) \ar@/_0.5pc/[d]_{F_0(a)}\ar@/^0.5pc/[r]^{\eta_x} & F_1(x)\ar@/_0.5pc/[d]_{F_1(a)} \\
& y
 &&
	F_0(y) \ar@/^0.5pc/[r]_{\eta_y}& F_1(y) 
	\ar@{->}^{\displaystyle{F_0,F_1}}"a";"b" }
$$
\end{defi}

We also  introduce the following basic tool, that will be usefull in order to understand natural transformations.
\begin{defi}\label{def:discreteinterval}
We call \emph{discrete interval}  the pair groupoid over $\{0,1\}$, and denote it by $\II$.
\end{defi}
More precisely, the discrete interval is the groupoid:
$$\II:=\left\{0 \,\substack{{\rightarrow} \\ {\leftarrow}}\,1\right\}$$ 
with $\{0,1\}$ as space of objects, and whose only non identical morphisms are given by a single arrow $\tau_{1,0}:0\to 1$ together with its inverse $\tau_{0,1}:1\to 0$. 
	 
 In the theory of categories, the discrete interval plays a similar role  as the usual interval $I=[0,1]$ in basic homotopy theory, as we now explain. Although we are only interested in groupoids, the following discussion can  easily be adapted to categories.

One may consider the product groupoid $\sC\times \II$, and notice that it contains two copies of $\sC$. More precisely, there are two obvious embeddings:
$$i_0,i_1:\sC \hookrightarrow \sC\times \II,$$
that send an object $x$ in $\sC$ to $(x,0)$ (respectively,  to $(x,1)$) and to a morphism $a$ in $\sC$ to $a\times \id_0$  (respectively, to $a\times \id_1$) where $\id_0,\id_1$ denote the identities in $\II$.

\begin{prop}\label{prop:naturaltandhomotopies}
Given two functors $F_0,F_1:\sC\to \sD$, there is one-to-one correspondence between:
\begin{itemize}
\item natural transformations $\eta: F_0\Rightarrow F_1$,
\item functors $F:\sC\times \II\to \sD$ such that $F_0=F\circ i_0$ and $F_1=F\circ i_1$.
\end{itemize}
\end{prop}
\begin{proof}
Given  $\eta: F_0\Rightarrow F_1$, the restrictions $F|_{i_0(\sC)}$ and $F|_{i_0(\sC)}$ being imposed, all we need in order to define $F$ is to specify the images of the morphisms $\id_x\times \tau : (x,0)\to (x,1)$ and $\id_x\times \tau^{-1} : (x,1)\to (x,0)$. This is easily done by setting:
\begin{align*}
F(\id_x\times \tau_{1,0})&=\eta_x ,    \\
F(\id_x\times \tau_{0,1})&=\eta_x^{-1}.
\end{align*}
We leave it to the reader to check that the conditions for $\eta$ to define a natural transformation imply that $F$ is indeed a functor.

Reciprocally, given a functor $F:\sC\times \II\to \sD$, we easily check that the above formulas define a natural transformation for $\eta:F_0\Rightarrow F_1$. 

The functor $F$ and how it relates with $\eta$ can be pictured as follows:
$$\SelectTips{cm}{}
\xymatrix@R=30pt@C=35pt{
	\ar@{}[r]_{\displaystyle{\sC \times \II\quad}}  
	&
	&& \ar@{}[r]_{\displaystyle{\sD}}&\\
	(x,0) \ar@/_0.5pc/[d]_{a\times \id_0}\ar@/^0.5pc/[r]^{\id_x\times \tau_{0,1}} & (x,1)\ar@/_0.5pc/[d]^{a\times \id_1}
	\ar@{{}{ }{}}@/^2,8pc/[r]^{}="a"
	&\ar@{{}{ }{}}@/^2,8pc/[r]^{}="b"&
	F_0(x) \ar@/_0.5pc/[d]_{F_0(a)}\ar@/^0.5pc/[r]^{\eta_x} & F_1(x)\ar@/_0.5pc/[d]^{F_1(a)}
	\\
	(y,0) \ar@/^0.5pc/[r]^{\id_x\times \tau_{0,1}}& (y,1)
	&&
	F_0(y) \ar@/^0.5pc/[r]^{\eta_y}& F_1(y)
	\ar@{->}^{\displaystyle{F}}"a";"b"}$$
Notice that the diagram on the left hand side always commutes because of the multiplication on the product  $\sC\times \II$. By functoriality, the diagram of the right hand side is also commutative. In fact, this point of view explains the commutativity condition in the Definition \ref{def:naturaltransf} of a natural transformation.
\end{proof}
\begin{remark}\label{rem:defnatural}
Although the Definition \ref{defnatural} is the one usually appearing in textbooks, one may argue that the description obtained from the Proposition~\ref{prop:naturaltandhomotopies} is more natural from the categorical point of view, in the sense that it spells out the notion of natural transformation purely in terms of functors. It will turn out to be our starting point in order to understand natural transformations in the smooth setting. This idea of natural transformation being "discrete" homotopies between functors is not original, see for instance the comments in \cite[\S 6.5]{RBrown}.
\end{remark}

 Moving to the smooth settings, we now consider two Lie groupoids $\G_M\tto M$ and $\G_N\tto N$, together with Lie groupoid morphisms $F_0,F_1:\G_M\to \G_N$. We denote by $\phi_0:M\to N$ and $\phi_1:M\to N$ the corresponding  smooth maps induced on the bases. 

There is no subtlety in the definition of a natural transformation between Lie groupoid morphisms. We just adapt the definition to the differential geometric setting (in other words, we internalize natural transformations in the category of differentiable manifolds).
\begin{defi}
A (smooth) \emph{natural transformation} $\eta: F_0\Rightarrow F_1$ is a smooth map $\eta: M \to \G_N $
such that $F_1(a)\circ \eta_x=\eta_y\circ F_0(a)$ for any $a\in \G_M$. 
\end{defi}
One may notice that the discrete interval $\II\tto \{0,1\}$  endowed with the discrete topology has a natural structure of a \emph{Lie} groupoid. Furthermore, the Proposition~\ref{prop:naturaltandhomotopies} immediately generalizes to Lie groupoids, as follows.

\begin{prop}\label{prop:naturaltandhomotopiessmooth}
Given two Lie groupoid morphisms $F_0,F_1:\G_M\to \G_N$, there is one-to-one correspondence between:
\begin{itemize}
\item Natural transformations $\eta: F_0\Rightarrow F_1$,
\item Lie groupoid morphisms $F:\G_M\times \II\to \G_N$ such that $F_0=F\circ i_0$ and $F_1=F\circ i_1$.
\end{itemize}
\end{prop}

\begin{remark}\label{rem:infinitesimalnaturaltransf}
The above proposition puts the notion of natural transformation in such a perspective that makes it obvious its infinitesimal counterpart, namely that there is none. 

Let us spell this out: in the theory of Lie groupoids, the \emph{infinitesimal counterpart} is obtained by applying the $\Lie$ functor.
 By doing so to a Lie groupoid morphism of the form $F:\G_M\times \II\to \G_N$ as in Proposition \ref{prop:naturaltandhomotopiessmooth}, we obtain a Lie algebroid morphism:
$$\Lie(F):\Lie(\G_M\times \II)\to \Lie(\G_N).$$
However, it is easily seen that $\Lie(\G_M\times \II)$ canonically identifies with the disjoint union of two copies of $A_M:=\Lie(\G_M)$, namely we have:
$$\Lie(\G_M\times \II)=A_M \amalg A_M.$$ 
This is due to the fact that the space of arrows in $\II$ come with a  discrete topology. As a consequence, $\G_M\times \II$ does not have connected source-fibers, even if $\G_M$ does. Hence, when applying the $\Lie$ functor to $F$, all information about the natural transformation $\eta$ is lost. In other words, from this point of view there is no {infinitesimal} counterpart to a single natural transformation.
\end{remark}

\begin{remark}
Natural transformations also have a nice interpretation in terms of generalized morphisms between Lie groupoids that we believe is worth pointing out. It that goes as follows.

 We first observe that the projection $\pi:\G_M\times \II \to \G_M$ is a Morita map. This comes from the fact that $\G_M\times \II$ can be obtained as a pullback groupoid. For this we consider the open cover of $M$ given by two open subsets, both equal to $M$. Denote by $p:M\coprod M\to M$ the projection, and by $p^!\G_M$ the corresponding pullback groupoid. Then there is an obvious identification $\G_M\times \II\simeq p^!\G_M$.
 
 This way, one can think of a natural transformation $F:\G_M\times \II\to \G_N$ as a generalized map $\G_M \dashrightarrow \G_N$. Indeed, such maps are preciseley defined as triples $(\widehat{\G}_M,\pi, F)$ where $\pi:\widehat{\G}_M \to  \G_N$ is a Morita map, and $F:\hat{\G}_M\to \G_N$ is a Lie groupoid morphism, as illustrated in the diagram below:
 $$
 \SelectTips{}{}
 \xymatrix{
 \G_M\times \II\ar[dr]^{F}\ar[d]^{\pi} &                      \\
 \G_M \ar@{-->}[r]_{(\widehat{\G}_M,\pi, F)} & \G_N.
 }
 $$
 The importance of such generalized morphisms comes from the fact that allow to build the localization of $\mathsf{LieGpd}$ along Morita maps. More concretely, one can think of the triple $(\widehat{\G}_M,\pi, F)$ as a composition $F\circ (\pi)^{-1}$ even though $\pi$ has no inverse in $\mathsf{LieGpd}$.
 
\end{remark}

\section{Natural homotopies}
The discussion of the previous section, and in particular the Remark \ref{rem:infinitesimalnaturaltransf}, seems to make hopeless the existence of an infinitesimal object that that would integrate to a  natural transformation. However, as suggested by the Proposition~\ref{prop:naturaltandhomotopies}, the discrete interval  $\II$ is a mere algebraic model for homotopies between functors. Yet, in the context of Lie groupoids, there is another obvious candidate to play this role, namely the pair groupoid $\bI:=(I\times I\tto I)$ over the interval $I:=[0,1]$. We shall denote by $\tau_{t_1,t_0}:t_0\to t_1$ the arrows in $\bI$.

 Note that that $\bI$ comes indeed with a natural structure of a Lie groupoid. Furthermore, given a Lie groupoid $\G_M\tto M$, we now have a whole family of embeddings parametrized by $t\in I$, as follows:
 $$i_t:\G_M \hookrightarrow \G_M\times \bI.$$
 Notice also that there is also an obvious embedding $\II \hookrightarrow \bI$, so that we can identify $i_t$ for $t=0,1$ with the embeddings into $\G_M\times \II$ defined in  the previous section.
 
With this smooth model of an interval at hand, the Proposition~\ref{prop:naturaltandhomotopies} suggests to adapt our approach by defining natural \emph{homotopies} as follows.
  
\begin{defi}\label{def:nauralhomotopygroupoid}
	Given two Lie groupoid morphisms $F_0,F_1:\G_M\to \G_N$, a \emph{natural homotopy} $ F_0\to F_1$ is a Lie groupoid morphism
	$$F:\G_M\times \bI \to \G_N.$$
	such that $F_0=F\circ i_0$ and $F_1=F\circ i_1$.
\end{defi}

 By an analogue of the Proposition \ref{prop:naturaltandhomotopiessmooth}, associated to a natural homotopy, rather than a single natural transformation, we now recover a whole smooth family of them.
\begin{prop}\label{prop:naturaltandhomotopiessmooth}
Given two Lie groupoid morphisms $F_0,F_1:\G_M\to \G_N$, there is one-to-one correspondence between:
\begin{itemize}
\item natural homotopies $F:\G_M\times \bI\to \G_N$ such that $F_0=F\circ i_0$ and $F_1=F\circ i_1$.
\item smooth families of natural transformations $\eta_t: F_0\Rightarrow F_t$.
\end{itemize}
\end{prop}
\begin{proof}
The proof is analogous to that of Proposition  \ref{prop:naturaltandhomotopies}.
Given $F:\G_M\times \bI\to \G_N$, we define family of functors  $F_t:=F\circ i_t:\G_M\to \G_N$, and an application $\eta:M\times I \to \G_N$ as follows:
\begin{align*}
\eta_{(x,t)}:=F(\id_x\times \tau_{t,0}),   
\end{align*}
As illustrated below, we have:	
$$\SelectTips{cm}{}
\xymatrix@R=30pt@C=35pt{
	         \ar@{}[r]_{\displaystyle{\G_M \times \bI\quad}}  
	            &
	&& \ar@{}[r]_{\displaystyle{\G_N}}&\\
	(x,0) \ar@/_0.5pc/[d]_{a\times \id_0}\ar@/^0.5pc/[r]^{\id_x\times \tau_{0,t}} & (x,t)\ar@/_0.5pc/[d]^{a\times \id_t}
	 \ar@{{}{ }{}}@/^2,8pc/[r]^{}="a"
	&\ar@{{}{ }{}}@/^2,8pc/[r]^{}="b"&
	F_0(x) \ar@/_0.5pc/[d]_{F_0(a)}\ar@/^0.5pc/[r]^{\eta^t_x} & F_t(x)\ar@/_0.5pc/[d]^{F_t(a)}
	 \\
	(y,0) \ar@/^0.5pc/[r]^{\id_x\times \tau_{0,t}}& (y,t)
    &&
	F_0(y) \ar@/^0.5pc/[r]^{\eta^t_y}& F_t(y)
\ar@{->}^{\displaystyle{F}}"a";"b"}$$
The diagram on the left hand side commutes  as a consequence of the definition of the multiplication in the product $\G_M\times \bI$, therefore, the one on the right hand side commutes by functoriality of $F$. It follows that $\eta_t:F_0\Rightarrow F_t$ is a natural transformation for all $t\in I$. Furthermore, $\eta$ is clearly smooth as a map $\eta:M\times I\to \G_N$.

Reciprocally, given a smooth family of natural transformations $\eta_t:F_0\Rightarrow F_t$. We define $F$ as follows:
$$ F(\id_x\times \tau_{t_2,t_1}):=\eta_{(x,t_2)}\cdot \eta_{(x,t_1)}^{-1} $$
for all $x\in M$. It is then straightforward to check that $F$ defines a Lie groupoid morphism $\G_M\times \bI\to \G_N$ as a consequence of $\eta_t$ being a natural transformation for all $t\in I$.
\end{proof}

\section{Infinitesimal counterpart of a natural homotopy}\label{sec:infinitesimal}
The Lie groupoid $\bI\tto I$  has associated Lie algebroid given by the tangent bundle:
$$\Lie(\bI)=(TI\to I).$$
With the Remark \ref{rem:infinitesimalnaturaltransf} in mind,  there is now an obvious infinitesimal counterpart of a natural homotopy, obtained by applying the $\Lie$ functor to the Definition \ref{def:nauralhomotopygroupoid}, as follows.
\begin{defi}
	Given two Lie algebroid morphisms $\Phi_0,\Phi_1:A_M\to A_N$, a \emph{natural homotopy} $ \Phi_0\to \Phi_1$ is a Lie algebroid  morphism
	$$\Psi:A_M\times TI \to A_N.$$
	such that $\Phi_0=\Psi\circ i_0$ and $\Phi_1=\Psi\circ i_1$, where $i_0,i_1:A_M\to A_M\times TI$ denote the obvious injections. 
\end{defi}

\begin{ex}
An $A$-path $adt:TI \to A$ between $x:=p_A\circ a(0)$ and $y=p_A\circ a(1)$ can be seen as a natural homotopy between the Lie algebroids morphisms $0_x:0_{\{*\}}\to A$ and $0_y:0_{\{*\}}\to A$. Here, $0_{\{*\}}$ denotes the trivial Lie algebroid over a point, and $0_x,\ 0_y$ are the obvious Lie algebroid maps associated with $x$ and $y$.
\end{ex}
 
Below, we shall look at natural homotopies between Lie algebroids morphisms in more details, part of the discussion is based on  \cite{BojKotovStrobl}, though our approach is slightly different, avoiding graphs. Along the text, we shall abuse notations, in particular $\partial t$ may denote indistinctly the obvious section of either vector bundle $TI\to I$ or $A_M\times TI\to M\times I$. Similarly, we shall not distinguish  $dt\in \Omega^1(TI)$ from its pullback along the projection $A_M\times TI \to TI$. We also use rather obvious notions (and notations) for "smooth families" of sections, of maps, etc... Although elementary, these notions turn out to be fundamental for the theory (see Remark \ref{rem:timedependentderivation}) and we shall refer to the Appendix \ref{sec:timedependence} for clarifications.

\begin{lemma}\label{lem:psidecvectorbundle}

 Any vector bundle map $\Psi:A_M\times TI \to A_N$ writes uniquely under the form:
\begin{equation}\label{eq:Psi=phi+theta}
\Psi:=\Phi +\theta dt
\end{equation} 
where $\Phi_t:A_M \to A_N$ is a smooth family of vector bundle maps covering $\phi_t:M\to N$, and  $\theta_t \in \Gamma(\phi_t^*A_M)$ a smooth family of sections of $A_N$ supported by $\phi_t$. Here $\phi: M\times I\to N$ is the map underlying $\Psi$.

 Reciprocally, any pair $(\Phi_t,\theta_t)$, where $\Phi_t: A_M\to A_N$ is a vector bundle map, and $\theta_t\in \Gamma(\phi_t^*A_N)$ is a smooth family of sections supported by the base map $\phi_t:M\to N$ of $\Phi_t$ induces by \eqref{eq:Psi=phi+theta} a vector bundle map $\Psi:A_M\times TI\to A_N$.
\end{lemma}

\begin{proof}[Proof of lemma \ref{lem:psidecvectorbundle}]
The vector bundle $A_M\times TI\to M\times I$ is obtained as a the Whitney sum of vector bundles over $M\times I$, as follows:
\begin{equation}\label{eq:AMTIdecomposition}
A_M\times TI=(A_M\times I)\oplus (M\times TI).
\end{equation} 
Further, a vector bundle map $A_M\times TI\to A_N$ covering a smooth map $\phi:M\times I\to N$ is the same as a vector bundle map $A_M\times TI\to \phi^*A_N$ covering the identity of $M\times I$. Decomposing $\Psi$ into each factor in \eqref{eq:AMTIdecomposition} the result easily follows,  namely we set $\Phi:=\Psi|_{A_M\times I}$, and $\theta:=\Psi\circ \partial t$ so that $\Psi|_{M\times TI}=\theta dt$. The converse statement should now be obvious.
\end{proof}

As explained in \S \ref{sec:LiealgasDGAs}, Lie algebroid morphisms are better described in terms of DGAs. The aim of the discussion below is to  derive the conditions, on $\Phi$ and $\theta$ in order for $\Psi=\Phi+\theta dt$ to define a Lie algebroid morphism in terms of DGAs.  
We shall fix a smooth family of vector bundle maps $\Phi_t:A_M\to A_N$ as in section \S \ref{sec:timedependence}, Recall that we denote by $\phi_t:M\to N$ the underlying family of smooth maps, and by  $\Phi^*:\Omega(A_N)\to \Omega(A_M)^I$ the morphism of graded algebras induced on forms. Here, $\Omega(A_M)^I$ denote the space of time dependent forms on $A_N$, defined as:
$$\Omega(A_N)^I:=\Omega(A_N\times I),$$
where $A_N\times I$ is seen as a vector bundle over $N\times I$.
\begin{prop}\label{prop:psiliealgebroidmorphism}
Let $\Phi_t:A_M\to A_N $ be a smooth family of vector bundle maps covering  $\phi_t:M\to N$,  and $\theta_t\in \Gamma(\phi_t^* A_N)$ a smooth family of sections of $A_N$ supported by $\phi_t$.

 Then $\Psi:=\Phi+\theta dt:A_M\times TI\to A_N$ defines a Lie algebroid morphism if and only if the following conditions hold for any $t\in I$:
 \begin{enumerate}[i)]
 	\item  $\Phi_t:A_M\to A_N$ is a Lie algebroid morphism,
 	\item the $\Phi_t^*$-derivation $i_{\theta_t}^{\Phi_t}$ as in Prop. \ref{prop:derivationandtheta} satisfies:
 	\begin{equation}\ddt \Phi^*_t= d_{A_M}\circ i_{\theta_t}^{\Phi_t}+  i_{\theta_t}^{\Phi_t} \circ d_{A_N},  \label{eq:Phimorphissm1}
 	\end{equation}
 \end{enumerate}
\end{prop}

\begin{proof}
The conditions $i)$ and $ii)$ are direct consequences of the fact that the cochain complex $\Omega(A_M\times TI)$ is bigraded (which in turn, follows from $A_M\times TI$ being a direct product). For clarity, we give below a more pedestrian proof.

It is easily seen that any $k$-form $\eta\in \Omega^k(A_M\times TI)$ decomposes uniquely as
$\eta=\alpha + \beta{}_\wedge dt$, where $\alpha \in \Omega^k(A_M)^I$ and $\beta \in \Omega^{k-1}(A_M)^I$. In other words, there is a canonical isomorphism:
\begin{align}
\label{eq:decompositionkforms}
\Omega^k (A_M\times TI)&=\Omega^k(A_M)^I \oplus\ \Omega^{k-1}(A_M)^I\! {}_\wedge dt.
\end{align}
Using this decomposition, one checks that both $d_{ A_M\times TI}$ and $\Psi^*$ can be decomposed in the following way:
\begin{align}
d_{A_M\times TI}&= d_{A_M\times I} + (-1)^{|\ \,|}\ddt(\ \,){}_\wedge dt, \label{eq:decompositiondAxTI} \\
\Psi^*&=\Phi^*-(-1)^{|\ \,|}i_\theta^{\Phi}{}{}_\wedge dt.\label{eq:decompositionPsiAxTI}
\end{align}
More precisely, for any $\alpha\in \Omega^k(A_M)^I$ and $\beta\in \Omega^{k-1}(A_M)^I$ we have:
\begin{align*}
d_{A_M\times TI}\bigl(\alpha + \beta_\wedge dt\bigr) 
&= d_{A_M\times I}(\alpha) + \Bigl( (-1)^k\ddt\alpha+d_{A_M\times I}\beta\Bigr){}_\wedge dt,
\end{align*}and for any $\alpha_N\in \Omega^l(A_N)$ we have:
\begin{align*}
\Psi^*\bigl(\alpha_N) &=\Phi^*(\alpha_N)-(-1)^{l} i_\theta^\Phi(\alpha_N){}_\wedge dt. 
\end{align*}
Here, $d_{A_M\times I}$ is the obvious Lie algebroid differential on $A_M\times I$ (see \eqref{prop:timedependentdifferential} for more details).
Furthermore, the operator $i_\theta^\Phi$ is given by the obvious  time-dependent version \eqref{eq:leibnizopreatordeftimedependent} of the formula \eqref{eq:leibnizopreatordef} describing $i_\theta^\Phi$ as a $\Phi^*$-derivation.

 Substituting the equations \eqref{eq:decompositiondAxTI}  and  \eqref{eq:decompositionPsiAxTI} in the condition $d_{A_M\times TI}\circ \Psi^* = \Psi^*\circ d_{A_N}$ for $\Psi$ to define a Lie algebroid morphism, and then taking into account the uniqueness of the decomposition of a form in \eqref{eq:decompositionkforms}, we deduce that $\Psi$ is a Lie algebroid morphism if and only if:
\begin{align*}
 d_{A_M\times I}\circ \Phi^* &= \Phi^*\circ  d_{A_N},\\
 \ddt\Phi^*&=  d_{A_M\times I}\circ i_{\theta}^{\Phi}+i_{\theta}^ {\Phi}\circ d_{A_N}.
\end{align*}
One concludes by noticing that both equations hold if and only if they hold  timewise. For the first equation, this follows more precisely from Lemma \ref{lemm:timedependentliealgmorphisms} and  \eqref{prop:timedependentdifferential}. Similarly for the second equation, and using obvious notations, one can see $\cV:=i_{\theta}^\Phi$ as a smooth family $\cV_t: \Omega^\bullet(A_N)\to \Omega^{\bullet+k}(A_M)$ of $\Phi^*_t$-derivations, to which the Prop. \ref{prop:derivationandtheta} associates a smooth family $\theta_t\in \Gamma(\phi_t^*A_M)$ supported by $\phi_t$.
\end{proof}
\begin{remark}
The easiest way to obtain the decomposition  \eqref{eq:decompositiondAxTI} is probably to check that it holds in local coordinates (see also \cite{BojKotovStrobl}). Note however that all the decompositions hold globally. In fact, the term $(-1)^{|\ \,|}\ddt(-)\wedge dt$ in the equation \eqref{eq:decompositionPsiAxTI} is just a convenient way to write the obvious operator $ d_{M\times TI}$ induced by $ d_{TI}$. This can be better understood using the formalism from graded geometry, where the right hand term in \eqref{eq:decompositiondAxTI} corresponds simply to the sum of the two homological vector fields $ d_{A_M}$ and $ d_{TI}$ on the product of graded manifolds $A_M[1]\times TI[1]$.
\end{remark}

\begin{remark}
In the terminology of \cite{BojKotovStrobl}, the Prop. \ref{prop:psiliealgebroidmorphism} essentially says that a natural homotopy is a homotopy of Lie algebroid morphisms by \emph{gauge transformations}. The precise relation between gauge transformations and smooth natural transformations, however, was not established in \cite{BojKotovStrobl}, neither its time-dependent version made explicit, which motivated part of this work. 
\end{remark}

The Theorem below, which essentially rephrases Prop. \ref{prop:psiliealgebroidmorphism}, summarizes the discussion of this subsection.

\begin{thm}
Let $\Phi_0,\Phi_1:A_M\to A_N$ be two Lie algebroid morphisms. We denote by  $\Phi_0^*,\Phi_1^*:(\Omega(A_N),\wedge, d_{A_N})\to (\Omega(A_M),\wedge, d_{A_M})$ the induced morphisms of DGAs.

Then $\Phi_0$ and $\Phi_1$ are naturally homotopic iff there exists:
\begin{itemize}
\item a smooth family of morphism of DGAs  $\Phi_t^*:(\Omega(A_N),\wedge, d_{A_N})\to (\Omega(A_M),\wedge, d_{A_M})$ extending $\Phi_0^*$ and $\Phi_1^*$,
\item a smooth family of $\Phi^*_t$-derivations $\cV_t:\Omega^\bullet(A_N)\to \Omega^{\bullet-1}(A_M)$ such that:
\begin{align}\label{eq:homotopycondition1}
 \ddt \Phi_t=  d_{A_M}\circ \cV_t +\circ \cV_t\circ  d_{A_N}.
 \end{align} 
\end{itemize}
\end{thm}

\begin{remark}[Flowing a Lie algebroid map by a time-dependent section]\label{rem:flowing}

Given a Lie algebroid morphism $\Phi_0:A_M\to A_N$, together with a time-dependent section $\hat{\theta}_t\in \Gamma(A_N)$ of $A_N$,  it is possible to flow $\Phi_0$ into a time-dependent family of Lie algebroid morphisms. 

In order to do this, we define a family of Lie algebroid mosphisms:
$$\Phi_t:=\Phi_{t,0}^{\hat{\theta}}\circ \Phi_0,$$
where $\Phi_{t,0}^{\hat{\theta}}:A_N\to A_N$ denotes the flow induced by the time-dependent derivation $D_{\hat{\theta}_t}:\Gamma(A_N)\to \Gamma(A_N),\ a \mapsto [\hat{\theta}_t,a]_{A_N}$ of $A_N$ (see for instance the Apendix in \cite{CrFe2} for how the flow of a derivation is obtained). Here, of course, we need to assume that the flow of the time-dependent vector field $\sharp_{A_N}(\hat{\theta}_t)$ defined on $N$ is defined up to time $t=1$. 

If we denote by $\phi_{t,0}^{\hat{\theta}_t}:N\to N$ the base map of $\Phi_{t,0}^{\hat{\theta}_t}$, and then define $\theta_t(m):=\hat{\theta}_t(\phi_t(m))$, then it is easily checked that the condition \eqref{eq:Phimorphissm1} in the Proposition \ref{prop:psiliealgebroidmorphism} is satisfied, so that $\Psi:=\Phi+\theta dt$ defines a natural homotopy between $\Phi_0$ and $\Phi_1$.

Notice yet how restrictive this construction is, in particular due to the fact that the base maps $\Phi_t:M\to N$ are obtained by composing $\phi_0$ with a diffeomorphism of $N$ (namely, the flow of the time-dependent vector field $\sharp_{A_N}\hat{\theta}$). This procedure \emph{does not} allow to perform and integrate basic homotopies such as those in \S\ref{sec:homotopiesofsmoothmaps}. For instance if $A_M=A_N=T\RR^n$, and $\Phi_t(v)=d\phi_t(v)$ where $\phi_t(v):=tv, v\in \RR^n$ is the usual contraction, then it is not possible to produce a time-dependent vector field $\hat{\theta}_t$ as above, simply due to the fact that for $t=1$, $\phi_1$ is not a diffeomorphism. It is however easy to fit $\Phi_t$ into a natural homotopy as explained in \S\ref{sec:homotopiesofsmoothmaps}. Besides the completeness issues, this raises serious restrictions from the homotopic point of view.
\end{remark}

\subsection{Operations on natural homotopies}

Similarly as natural transformations, natural homotopies can be composed both horizontally and vertically (up to minor smoothness issues). In this section we briefly describe those operations.

\begin{defi} Given natural homotopies $\Psi=\Phi+\theta dt: A_M\times TI\longrightarrow A_N$ and $\Psi^\prime=\Phi^\prime+\theta^\prime dt:A_N\times TI\longrightarrow A_P$, we define the \emph{horizontal composition of $\Psi^\prime$ and $\Psi$} as the natural homotopy:
\begin{align*}
\Psi^\prime\circ_H \Psi=\Phi^H +\theta^H dt:A_M\times TI\longrightarrow A_P
\end{align*}
whose components $\Phi^H$ and $\theta^H$ are given as follows:
 \begin{align*}
\left\{\begin{array}{lcl}
\Phi^H=\Phi^\prime\circ \Phi,\\
\theta^H = \Phi^\prime\circ \theta+\theta^\prime \circ \phi,
\end{array}\right.
\end{align*}
where $\phi:M\times I\to N$ denotes the base map of $\Phi:A_M\times I\to A_N$. 
\end{defi}

Of course, the previous equalities mean that:
$\Phi^H_t=\Phi^\prime_t\circ \Phi_t$ and $ \theta_t^H=\Phi_t^\prime\circ\theta_t+\theta_t^\prime\circ\phi_t,
$ for every $t\in I$.

\begin{remark}
In fact, the horizontal composition can be understood as a usual composition of maps, as follows. Any Lie algebroid $\Psi:A_M\times TI\to A_N$ gives rise to a Lie algebroid morphism $\widehat{\Psi}:A_M\times TI\to A_N\times TI$ that commutes with the projections on the $TI$ factor (in other words, by setting $\widehat{\Psi}(a,l):=(\Psi(a),l)$) and reciprocally. The horizontal composition is then defined so that $\widehat{\Psi'\circ_H\Psi}=\widehat{\Psi'}\circ\widehat{\Psi}$.

This makes it straightforward to check that $\Psi'\circ_H\Psi$ is indeed a Lie algebroid morphism. This also suggests that $A_M\times TI$ should really be thought of as a trivial fibration $A_M\times TI\to TI$, and a natural homotopy as a morphism between two such fibrations, covering the identity on $TI$.
\end{remark}

The vertical composition is obtained by a concatenation-type procedure, in the following way.

\begin{defi}
Given natural homotopies $\Psi=\Phi+\theta dt:A_M\times TI\longrightarrow A_N$ and $\Psi^\prime=\Phi^\prime+\theta^\prime dt: A_M\times TI\longrightarrow A_N$ such that $\Phi_1=\Phi'_0$, the \emph{vertical composition of $\Psi^\prime$ and $\Psi$} is defined as the natural homotopy 
\begin{align*}
\Psi^\prime\circ_V \Psi=\Phi^V+\theta^V dt:A_M\times TI\longrightarrow A_N
\end{align*}
whose components $\Phi^V$ and $\theta^V$ are given as follows:
\begin{align*}
\Phi^V_t&=\left\{\begin{array}{lcl}
\Phi_{2t} & \textrm{if}& 0\leq t\leq 1/2,\\
\Phi_{2t-1}^\prime & \textrm{if}& 1/2\leq t\leq 1,
\end{array}\right.
\end{align*} 
and:
\begin{align*}
\theta^V_t&=\left\{\begin{array}{lcl}
2\theta_{2t} & \textrm{if}& 0\leq t\leq 1/2,\\
2\theta_{2t-1}^\prime & \textrm{if}& 1/2\leq t\leq 1.
\end{array}\right.
\end{align*}
\end{defi}

\begin{remark}
Notice the factors $2$ appearing in the formula for $\theta^V$. In fact, these are due to the reparametrization in $t$. More precisely, the vertical composition is such that: 
\begin{align*}
\Psi'\circ_V\Psi|_{A_M\times[0,1/2]}&=\Psi\circ (\mathrm{id}_{A_M}\times d\sigma_1),\\ \Psi'\circ_V\Psi|_{A_M\times[1/2,1]}&=\Psi'\circ (\mathrm{id}_{A_M}\times d\sigma_2),
\end{align*}
 where $\sigma_1:[0,1/2]\to I$ is given by $\sigma_1(t):=2t$, and $d\sigma_1:T[0,1/2]\to TI$ denotes its differential, while  $\sigma_2:[1/2,1]\to I$ is given by $\sigma_2(t):=2t-1$, and $d\sigma_2:T[0,1/2]\to TI$ denotes its differential. It is when taking the differentials of $\sigma_1$ and $\sigma_2$ that appear the factors $2$. Also, this should make it clear that $\Psi'\circ_V\Psi$ is a Lie algebroid morphism.
\end{remark}

\begin{remark}
It should also be clear that the vertical composition $\Psi'\circ_V\Psi$ is \emph{not} smooth in general. However, provided $\Phi_1=\Phi'_0$, it is always possible to re-parametrize  the $t$-coordinate by using a cut-off function to re-parametrize, so as to obtain a smooth vertical composition. Details are not relevant for this work, and will be left to the reader  (see \cite{CrFe2, BZ, BO} where similar constructions are carried through).
\end{remark}

Notice also that, unlike the vertical composition of natural transformations, the vertical composition of natural homotopies is \emph{not} associative (not strictly, at least). The horizontal composition is, and furthermore, horizontal and vertical compositions are compatible in the following sense:
\begin{align*}
(\Psi_1^\prime\circ_H \Psi_1)\circ_V (\Psi_0^\prime\circ_H \Psi_0)=(\Psi_1^\prime\circ_V \Psi_0)\circ_H (\Psi_1\circ_V \Psi_0). 
\end{align*}
whenever all the above terms are well-defined. 
\section{Integration of natural homotopies}\label{sec:integration}

We now have enough tools in had in order to prove the main result of this work, which is to provide a integration procedure, in order to obtain smooth natural transformations.

\begin{thm}\label{thm:main}
Let $A_M$ and $A_N$ be Lie algebroids over smooth manifolds $M$ and $N$ respectively, together with a smooth family $\Phi_t:A_M\to A_N$ of Lie algebroid morphisms parametrized by $t\in [0,1]$. 

We denote by $\phi_t: M\to N$ the  underlying family of smooth maps, and by 
$$\Phi_t^*:\bigl(\,\Omega^\bullet(A_N), \wedge{},  d_{A_N}\,\bigr)\longrightarrow  \bigl(\,\Omega^\bullet(A_M), \wedge{},  d_{A_M}\bigr)$$ the smooth family of morphisms of DGAs induced by $\Phi_t$.

Assume that there exists a smooth family $\theta_t\in \Gamma(\phi_t^*A_N)$ of sections of $A_N$ supported by $\phi_t$ that satisfies the following condition:
\begin{equation}\label{eq:conditionphit}
\ddt\Phi_t^*=  d_{A_M}\circ i_{\theta_t}^{\Phi_t}+i^{\Phi_t}_{\theta_t}\circ d_{A_N},
\end{equation}
where $i_{\theta_t}^{\Phi_t}:\Omega^\bullet(A_N)\to \Omega^{\bullet-1}(A_M)$ is the 
 family of maps obtained  by contraction with $\theta_t$, namely:
\begin{equation}
\bigl\langle i_{\theta_t}^{\Phi_t} \alpha_N,v_1{}_\wedge\dots{}_\wedge v_{p-1}\bigr\rangle=\bigl\langle\alpha_N, \theta_t(m){}_\wedge \Phi_t(v_1){}_\wedge\dots{}_\wedge\Phi_t(v_{n-1})\bigr\rangle.
\end{equation}
for any $v_1,\dots,v_{n-1}\in (A_M)_m,$ and $m\in M$.

Then, the following assertions hold:
\begin{enumerate}[i)]
	\item The morphisms $\Phi_0^*$ and $\Phi_1^*$ are chain homotopic.
	\item Whenever $A_M$ and $A_N$ are integrable Lie algebroids, the assignment:
	\begin{equation}\label{eq:nattranspath}
	m\in M\mapsto \eta(m):=[\theta(m)dt]_{A_N}\in \G(A_N)
	\end{equation}
	defines a smooth natural transformation $\eta: F_0\Rightarrow F_1$, where
	$$F_0,F_1: \G(A_M)\to \G(A_N)$$
	denote the Lie groupoid morphisms integrating $\Phi_0$ and $\Phi_1$ respectively.
	
	Here, we denoted by $[\theta(m)dt]_{A_N}$ the $A_N$-homotopy class of the $A_N$-path $\theta(m)dt:TI\to A_N$.
\end{enumerate}
\end{thm}
\begin{proof}
The proof of the first assertion  is similar to the homotopy invariance of smooth maps in De Rham cohomology (see, for instance \cite[Prop.\,11.5, p.\,276]{Lee}). More precisely, given $\alpha \in \Omega(A_N)$, we check that using the condition \eqref{eq:conditionphit} we have:
\begin{align*}
\Phi_1^*\alpha-\Phi_0^*\alpha
&=\int_0^1\ddt\Phi_t^*(\alpha) dt\\
&=d_{A_M} \int_0^1 i_{\theta_t}^{\Phi_t}(\alpha) dt-\int_0^1 i_{\theta_t}^{\Phi_t}( d_{A_N}\alpha)dt
\end{align*}
so that the operator $\Theta:\Omega^\bullet(A_N)\to \Omega^{\bullet-1}(A_M)$ given by:
$$\Theta(\alpha):=\int_0^1 i_{\theta_t}^{\Phi_t}(\alpha)dt$$
defines indeed a chain homotopy between $\Phi^*_0$ and $\Phi^*_1$. Here, we used the fact that integration with respect to the time variable  commutes with the Lie alegbroid differentials (see \eqref{integrationoperator} and \eqref{eq:differentialcommuteswithint} for more details).
 
In order to prove the second assertion, we use the Proposition~ \ref{prop:psiliealgebroidmorphism}, according to which $\Psi:=\Phi+\theta dt$ defines a Lie algebroid morphism $\Psi:A_M\times TI \to A_N$
such that $\Phi_0=\Psi\circ i_0$ and $\Phi_1=\Psi\circ i_1$. Then, by applying the integration functor to  $\Psi$, we obtain a Lie groupoid morphism:
$$F:\G(A_M)\times \bI \to \G(A_N).$$
such that $F_0=F\circ i_0$ and $F_1=F\circ i_1$. Therefore there exist natural transformations  $\eta_t:F_0\Rightarrow F_t$ by the Proposition \ref{prop:naturaltandhomotopiessmooth}, in particular for $t=1$.

Finally, the explicit formula for the the natural transformation $\eta$ is a consequence of the formula of $\eta$ in the proof of the Proposition \ref{prop:naturaltandhomotopiessmooth}. Namely, it is easily seen that  the element $\id_x\times \tau_{1,0}$, as a homotopy class of $A_M\times TI$-path, is represented by the following path:
\begin{align*}
(0^{A}_x\times \partial t)dt:TI &\to A_M\times TI \\
\partial t&\mapsto \bigl(0^{A_M}_x,\partial t\bigr).
\end{align*} 
By the definition of the integration functor, 
the image of $\id_x\times \tau_{1,0}$ by $F$ is given by the homotopy class of the  $A_N$-path  $\Psi\circ (0^{A}_x\times \partial t)dt$. Then using the formula  $\Phi=\Phi+\theta dt$, we successively obtain:
\begin{align*}
\eta(x)&=F(\,\id_x\times \tau_{1,0}\,)\\
       &=F\bigl(\,[(0^{A}_x\times \partial t)dt]_{A_M} \bigr)\\
       &=\bigl [\Psi\circ (0^{A}_x\times \partial t)dt \bigr]_{A_N}\\
       &=[\theta(x,-)dt]_{A_N},      
\end{align*}
which completes the proof.
\end{proof}


\begin{remark}
It is easily checked that the integration procedure described in the previous Theorem commutes with the vertical and horizontal composition defined in section 
\end{remark}

\begin{remark}\label{rem:bissections}
One might bring the argument in Remark \ref{rem:flowing} (to  which the notations below refer) a bit further.

 By considering only homotopies induced by time-dependent sections $\hat{\theta}_t\in \Gamma(A_N)$ one is essentially considering the Lie algebroid $A_N$ through the prism of its infinite dimensional Lie algebra of sections. Say $A_N$ is integrable, then the "Lie group" associated with $\Gamma(A_N)$ is, at least heuristically, the group of bissections of $\G(A_N)$, which can be argued by looking at the "tangent space" at identities of the group of bissections. In fact,  it is not hard to see that the map $\Phi_{1,0}^{\hat{\theta}}:A_N\to A_N$ defined in Remark  \ref{rem:flowing} integrates to a map $\G(A_N)\to \G(A_N)$ which is obtained by conjugation in $\G(A_N)$  by a bissection, the which is obtained from  $\hat{\theta}$ (the precise construction was detailed in the Apendix of \cite{CrFe2}). 
 
 In other words, by only flowing along time-dependent sections, we essentially replaced a Lie algebroid with its Lie algebra of sections, and its corresponding Lie groupoid by its group of bissections, negating the very "-oid" nature of these structures (essentially, we set ourselves in the situation of the next Example \ref{ex:homotopy-A_N-Lie-algebras}). For higher degrees, the result is due to Vaintrob \cite{Vain}.
 
 While the usefulness of bissections is, in certain instances, undeniable, it is certainly surprising that, in the context of Lie groupoids, smooth natural transformations were so overlooked in the literature in the favour of bissections.
\end{remark}

\begin{ex}\label{ex:homotopy-A_N-Lie-algebras}
The only case where the point made in the previous Remark \ref{rem:bissections} is not relevant, which is the one of Lie algebras. Indeed, assume that $A_M$ and $A_N$ are Lie algebras, $A_M=\g$ and $A_N=\h$. We denote by $G$ and $H$ the source simply-connected Lie groups integrating $\g$ and $\h$ respectively. Given Lie algebras morphisms  $\Phi_0:\g\to \h$ and $\Phi_1:\g\to \h$, a natural homotopy $\Psi=\Phi+\theta dt$ simply involves a family $\theta_t$ of elements of $\mathfrak{h}$, and the corresponding $\mathfrak{h}$-path in Theorem \ref{thm:main} yields an element $h:=[\theta]_{\mathfrak{h}}$ in $H$.

In other words, in the case of Lie algebras, the Theorem \ref{thm:main} states that the Lie algebra morphisms $\Phi_0,\Phi_1$ are naturally homotopic \emph{iff} there exists $h\in H$ such that $\Phi_1=\text{Ad}_h\circ \Phi_0$.
\end{ex}

\section{Examples and Applications}
\subsection{Homotopies of smooth maps}\label{sec:homotopiesofsmoothmaps}
Our setting can be applied to the usual notion of homotopy between smooth maps in the following way

\begin{prop}
	Let $M$ and $N$ be smooth manifolds and $\phi_i:M\to N$, $t=0,1$ two smooth maps. Then smooth homotopies $\psi:M\times I \to N$ are the same thing are natural homotopies $\Psi:TM\times TI\to TN$.
\end{prop}
\begin{proof} This follows from the well-known fact that a vector bundle $\Psi:TM\times TI\to TN$ covering $\phi:M\times I\to N$ is a Lie algebroid morphism if and only $\Psi= d \psi$.
\end{proof}

In that case, the main Theorem  \ref{thm:main} reads as follows.

\begin{prop}\label{prop:fudamentalgroupoidfunctors}
Given smooth manifolds $M,N$ and a smooth homotopy $\phi:M\times I\to N$ between $\phi_0:M\to N$ and $\phi_1:M\to N$, the map $\eta: M\to \Pi(N)$ given by:
$$\eta(m):=[\gamma^{\phi}_{(m)}],$$
where $[\gamma^{\phi}_{(m)}]$ denote the homotopy class of the path $\gamma^{\phi}_{(m)}:I\to N, \ t\mapsto \phi_t(m)$ defines a smooth natural transformation:
$$\begin{gathered}\SelectTips{}{}
 \xymatrix@R=30pt@C=30pt{
 	\Pi(M)
 	\ar@/_1pc/[r]_{F_1}^{}="a"
 	\ar@/^1pc/[r]^{F_0}_{}="b"
 	& \Pi(N)
 	\ar@{=>}^{\eta}"a";"b"}
 \end{gathered}
$$
Here $\Pi(M)\tto M$ and $\Pi(N)\tto N$ denote the fundamental groupoids and $M$ and $N$, respectively, and $F_0,F_t:\Pi(M)\to \Pi(N)$ the Lie groupoid morphisms induced by $\phi_0$ and $\phi_1$ respectively.
\end{prop}

\begin{proof}
Given a smooth homotopy $\phi:M\times I\to N$ between $\phi_0:M\to N$ and $\phi_1:M\to N$, we obtain a smooth family of Lie algebroid morphisms $\Phi_t:=d\phi_t: TM\to TN$ that integrate to a  family of Lie groupoid morphisms:
$$F_t:\Pi(M)\to \Pi(N).$$
 Obviously, the differential $\Psi:= d \phi:TM\times TI \to TN$ defines a natural homotopy between $\Phi_0$ and $\Phi_1$. Furthermore, $\Phi$ is clearly of the form:
 $$\Phi_{x,t}:= d \phi_t+ \frac{\partial \phi}{\partial t} d t.$$
In other words in that case, we have:
$$\theta(m):=\frac{\partial \phi}{\partial t}.$$
The result then easily follows by applying our main Theorem. \ref{thm:main}.
\end{proof}

\begin{remark}
One may wonder what happens if we replace the fundamental groupoids $\Pi(M)$ and $\Pi(N)$ by the pair groupoids $M\times M\tto M$ and $N\times N \tto N$, that also integrate $TM$ and $TN$, respectively. However, by doing that, the Proposition \ref{prop:fudamentalgroupoidfunctors} becomes void, indeed it is easily checked that for \emph{any} Lie groupoid $\G\tto M$, and \emph{any} couple of Lie groupoid morphisms $F_0,F_1:\G\to N\times N$, there exists a unique natural isomorphism $\eta:F_0\Rightarrow F_1$ (it is  given by $\eta(x)=(F_0(x),F_1(x))\in N\times N$, for any $x\in M$).
\end{remark}

%


\subsection{Deformations retracts of Lie algebroids}
In the sequel, we consider a Lie algebroid  $A_M\to M$ together with  a \textbf{Lie subalgebroid} $A_R\to R$ (see \cite{MMrcun}). Recall that this means that we have an injective Lie algebroid morphism
 $$i_{A_R}:A_R\hookrightarrow A_M,$$ 
 such that the corresponding base map $i_R:R\hookrightarrow M$ is an embedding.
\begin{defi}\label{def.deformationretraction}
Given a Lie subalgebroid $i_{A_R}:A_R\hookrightarrow A_M,$ a  \textbf{(weak) deformation retraction} from $A_M$ to $A_R$ is a natural homotopy between $\Phi_0:=\id_{A_M}$ and a retraction from $A_M$ to $A_R$.
\end{defi}
More precisely, a deformation retraction is given by a Lie algebroid morphism $\Psi=\Phi+\theta dt:A_M\times TI \to A_M$, such that: 
 \begin{enumerate}[i)]
 \item for $t=0$, we have $\Phi_0=\id_{A_M}$, 
 \item for $t=1$, the map $\Phi_1:A_M\to A_M$ factor through a Lie algebroid morphism $\check{\Phi}_1:A_M\to A_R$ which is a retraction, namely we impose that:
 \begin{align*}
  i_{A_R}\circ \check{\Phi}_1&=\Phi_1,\\
  \check{\Phi}_1\circ i_{A_R}&=\id_{A_R}.
 \end{align*}
 \end{enumerate}
   The above commutation relations are illustrated by the  diagram:
   $$\SelectTips{}{}
   \xymatrix@R=25pt@C=25pt{                   A_M \ar[r]^{\Phi_1}
                                    \ar[dr]^-*!/d2pt/{\labelstyle\check{\Phi}_1} & A_M \\
              \overset{\ }{A_R}\ar@{^{(}->}[u]^{i_{A_R}}
                \ar[r]^<<<<{\id_{A_R}}&              \overset{\ }{A_R}\ar@{^{(}->}[u]_{i_{A_R}}
              }
   $$

We say that $A_R$ is a \textbf{deformation retract} of $A_M$ if there exists a deformation retraction from $A_M$ to $A_R$.

\begin{prop}
A Lie algebra $\g$ admits no deformation retract except from itself. More precisely, if $\g_R\subset\g$ is deformation retract of $\g$, then $\g$ and $\g_R$ are isomorphic as Lie algebras.
\end{prop}

\begin{proof}
If $A_M=A_N=g$ is a Lie algebra, then by the Example \ref{ex:homotopy-A_N-Lie-algebras}, $\Phi_1$ is necessarily invertible.
\end{proof}

\begin{thm}\label{thm:retract}
Let $A_M$ be an integrable Lie algebroid, and $\Psi=\Phi+\theta dt:A_M\times TI \to A_M$ be  a deformation retraction of $A_M$ to $A_R$. Then the following assertions hold:
\begin{enumerate}[i)]
\item The Lie algebroid $A_R$ is a integrable.
\item $\G(A_M)$ identifies, as a set theoretical groupoid, with the pullback groupoid $\phi_1^!\G(A_T)$, where $\phi_1$ denotes the restriction of the base map of $\Psi$ to $M\times\{1\}$, seen as a map $\phi_1:M\to R$.
\item if furthermore the base map $\check{\phi}_1:M\to R$ is transversal to the characteristic foliation of $A_R$, that is if: 
$$
\mathrm{Im}( d \check{\phi}_1)_m+\mathrm{Im}(\sharp_{A_R})_{\check{\phi}_1(m)}=T_{\check{\phi}_1(m)}R
$$
 for any $m\in M$, then the identification of item $ii)$ is a Lie groupoid isomorphism, in particular $\G(A_M)$ and $\G(A_R)$ are Morita equivalent.
\end{enumerate} 
\end{thm}
\begin{proof}
The item $i)$ is a well known result from Moerdijk and Mr\u{c}un \cite[\S 3.2]{MMrcun}, stating that any Lie subalgebroid of an integrable Lie algebroid is integrable. In order to prove $ii)$, one first checks that, for any Lie algebroid morphism  $\check{\Phi}_1:A_M\to A_R$ covering $\check{\phi}_1:M\to R$, the map:
\begin{align*}
K:\G(A_M)&\longrightarrow\  \check{\phi}_1^!\G(A_R)\\
[a_Mdt] &\longmapsto (\target([a_Mdt]),\check{\Phi}_1([a_M dt]),\source([a_Mdt]))
\end{align*}
always defines a (set) groupoid morphism. Here $[a_Mdt]$ denotes the $A_M$-homotopy class of any $A_M$-path $a_Mdt:TI\to A_M$, and $\check{\Phi}_1:\G(A_M)\to \G(A_R)$ is the Lie groupoid morphism integrating $\check{\Phi}_1:A_M\to A_R$.

One further checks that, if there is a deformation retraction $\Psi=\Phi+\theta dt$ as in Definition  \ref{def.deformationretraction},  then $K$ admits the following map as an inverse: 
\begin{align*}
L:\check{\phi}_1^!\G(A_R)\ &\longrightarrow \ \G(A_M)\\
(y,[a_Rdt],x)&\longmapsto \eta_y\cdot [i_{A_R}(a_R)dt]\cdot (\eta_x)^{-1} 
\end{align*}
Here, $\eta(m):=[\theta(m)dt]\in \G(A_M)$ is the natural transformation as in the equation \eqref{eq:nattranspath} of Theorem \ref{thm:main}.

The transversality condition in $iii)$ ensures that the pullback groupoid $\check{\phi}^!_1\G(A_N)$ is a Lie groupoid, then the isomorphisms $K$ and $L$ defined above are automatically smooth  (see for instance \cite[p. 204]{HM} for pullbacks of Lie groupoids and  algebroids).
\end{proof}

\begin{remark}
Given a Lie algebroid $A_M$, and a Lie subalgebroid $A_R\hookrightarrow A_M$ that is transverse to the characteristic foliation $\mathrm{im}\,\sharp_{A}\subset TM$, there are normal form results (see \cite{BLM, FM1,FM3}) that allow to describe the restriction of $A_M$ to a neighborhood of $R$ as a pullback, namely: $A_M|_\mathcal{U}\simeq p^!\!A_R$, where $p:\mathcal{U}\to R$ is a tubular neighborhood of $R$.

In particular, since  $p$ has contractible fibers, it is easily seen that $p^!A_R$ deformation retracts to $A_R$. In other words, any transversal Lie subalgebroid is a deformation retract of a sufficiently small tubular neighbourhood of his.

Not all deformation retracts are of this form though, and in fact the Theorem \ref{thm:retract} is of a \emph{global} nature. Actually, the fibers of a tubular neighborhood are diffeomorphic to $\RR^n$, while the fibers of a deformation retraction need only be contractible. An explicit example is given by the Whitehead manifold $W$, that retracts to a point, while it is not diffeomorphic to $\RR^3$. As a consequence the Lie algebroid $TW\times T\RR$ deformation retracts to $T\RR$, while it is not of the form $p^!T\RR$ for a tubular neighborhood of $\{0\}\times \RR$. What remains unclear though, is if any deformation retract is transversal.
\end{remark}

\appendix

\section{Lie algebroid maps and time dependence}\label{sec:timedependence}

In this work, we have to deal with various types of time-dependent geometric objects, like sections, vector bundle maps. Below we clarify what is meant by smoothness for such objects.

\begin{defi}\label{def:smoothfamilyofmaps}
	A \emph{smooth family of maps} $\phi_t:M\to N$ parametrized by $t\in I:=[0,1]$ is a smooth map $\phi:M\times I \to N$.
\end{defi}
We define and use similar notations for smooth families of vector bundle maps, of sections, and so on. In particular for forms, supported sections and derivations, we shall use the following notations.

\begin{defi}
Given a vector bundle $A_M\to M$, a \emph{time-dependent ${k}$-form} on ${A_M}$ is a section of the vector bundle $\wedge^k A^*_M\times I$ over $M\times I$.
\end{defi}

 We denote by:
$$\Omega(A_M)^I:=\Omega(A_M\times I).$$
the space of time-dependent forms. In this work, we interpret any $\alpha\in \Omega(A_M)^I$ as a smooth family of sections $\alpha_t\in \Omega(A_M)$, where $\alpha_t$ is obtained as a pullback, namely:
\begin{equation}\label{eq:def:alphat}
\alpha_t:=(i^{A_M}_t)^*(\alpha),
\end{equation}
 where  $i^{A_M}_t:A_M\to A_M\times I$ denotes the obvious injection, and $(i^{A_M}_t)^*:\Omega(A_M)^I\twoheadrightarrow \Omega(A_M)$ the map induced on forms. With these notations, since $(i^{A_M}_t)^*$ is a morphism of graded algebras, we have:
\begin{equation}
(\alpha{}_\wedge\beta)_t=(\alpha_t){}_\wedge (\beta_t).\label{prop:timedependentproduct}
\end{equation}

\begin{defi}
Given a smooth family of maps $\phi_t:M\to N$, and a vector bundle $A_N$ over $N$, a \emph{smooth family $\theta_t$ of sections of $A_N$ supported by $\phi_t$} is, by definition, a section $\theta$ of the pullback bundle $\theta\in \Gamma(\phi^*A_N)$, where $\phi$ is seen as a map $\phi:M\times I\to N$.
\end{defi} 

For a time-dependent vector bundle map, the notion of time-dependent $\Phi^*\!$-derivations is obtained as in the Definition  \ref{def:phiderivation}, with $A_M\times I$ rather than $A_M$. With  our notations, this reads as follows.

\begin{defi}\label{def:phiderivationtimedependent}
Consider  $\Phi_t:A_M\to A_N$ a time-dependent vector bundle map covering $\phi_t:M\to N$,  and denote by $\Phi^*:\Omega(A_N)\to \Omega(A_M)^I$ the induced map on forms.  A $\Phi^*\!$-derivation of degree $k$ is a degree $k$ map $$\cV: \Omega^\bullet(A_N)\to \Omega^{\bullet+k}(A_M)^I$$
 such that:
\begin{align}
\cV(\alpha+\beta)&=\cV(\alpha)+\cV(\beta),\\
\label{eq:phiderivationproperty}
\cV(\alpha_\wedge\beta)
&=\cV(\alpha)_\wedge\Phi^*(\beta)+(-1)^{k.\text{deg}(\alpha)}\Phi^*(\alpha)_\wedge\cV(\beta).
\end{align}
\end{defi}

\begin{remark}\label{rem:timedependentderivation}
Because of the equation \eqref{prop:timedependentproduct}, a $\Phi^*\!$-derivation $\cV$ is the same as a \emph{smooth} family $\cV_t$ of $\Phi_t^*\!$-derivations.

Note however that the terminology of a \emph{smooth family} here is a bit misleading, for both $\Phi^*$ and $\cV$. Indeed, we view for instance $\Phi^*$ as a mere algebraic object, namely a morphism of graded algebras $\Omega(A_N)\to \Omega(A_M)^I:=\Omega(A_M\times I)$. There is no topology involved in this definition, neither on $\Omega(A_N)$ nor on $\Omega(A_M)^I$. Rather, they are the algebraic properties of $\Phi^*$ that insure the smoothness of $\Phi$. In degree $0$ for example, the only condition is that $\phi^*:C^\infty(N)\to C^\infty(M\times I)$ be a morphism of commutative algebras, then it is well known that such morphisms are always pullbacks by a \emph{smooth} map $\phi: M\times I\to N$ (a result which was proved by Bkouche \cite{Bkouche}).

Similar considerations hold for $\cV$, let us spell out the correspondence $\cV \leftrightarrow \cV_t$ in order to avoid confusions. Given $\cV$
we obtain a smooth family $\cV_t$ by setting $\cV_t:=(i^{A_M}_t)^*\circ \cV$. Reciprocally, given a family of aplications $\cV_t:\Omega(A_N)\to \Omega(A_M)$, we set:
\begin{equation}\label{eq:nunut}
\bigl\langle \cV(\alpha_N)_{(m,t)}, a_1{}_\wedge\dots_\wedge a_k\bigr\rangle:=\bigl\langle\cV_t(\alpha_N)_{m}, a_1{}_\wedge\dots_\wedge a_k\bigr\rangle.
\end{equation}  
for any $\alpha_N\in \Omega^k(A_N)$ and $a_1,\dots,a_k\in (A_M\times I)_{(m,t)}$ (the latter being identified with $(A_M)_m $ on the right hand side of \eqref{eq:nunut}). Then, by definition, the family $\cV_t$ is \emph{"smooth"} if $\cV$ \emph{takes values} in $\Omega(A_M)^I$, the space of smooth sections of $\wedge A_M^*\times I\to M\times I$. 
\end{remark}

Assume now that $A_M\to M$ is a Lie algebroid. Then there is, on $A_M\times I$, an obvious structure of a  Lie algebroid over $M\times I$ such that for any $t\in I$, the injections $i^{A}_t:A_M\hookrightarrow A_M\times I$ define  Lie subalgebroids.  The corresponding differential on $\Omega(A_M)^I$ is characterized by:
\begin{align}
\label{prop:timedependentdifferential}
(\d_{A_M\times I}\,\alpha)_t&=\d_{A_M}(\alpha_t),
\end{align}
for any time dependent forms $\alpha,\beta\in \Omega(A)^I$, and any $t\in I$. 

Notice that $\d_{A_M\times I}$ is $C^\infty(I)$-linear, and that the obvious operator:
 $$\ddt:\Omega(A_M)^I \to \Omega(A_M)^I$$
is a degree $0$ derivation of $(\Omega(A_M)^I,{}\wedge)$ that commutes with $\d_{A_M\times I}$.

Likewise, time-dependent forms can be integrated with respect to the time variable, yielding an operator:
\begin{align}\label{integrationoperator}
\int_0^1:\Omega(A_M)^I&\longrightarrow \Omega(A_M)\\
\alpha\quad&\longmapsto \int_0^1\alpha_t dt.
\end{align}
The easiest way to define the above operator is by duality, as follows:
$$\big\langle \int_0^1\alpha_t dt, a_1{}_\wedge\dots {}_\wedge a_n\big\rangle:=\int_0^1 \langle\alpha_t,a_1{}_\wedge\dots{}_\wedge a_n\rangle dt.$$
for any $\alpha \in \Omega^n(A_M)$ and $a_i\in \Gamma(A_M)$, $i=1\dots n$. Here, on the right hand side, one just integrate a function on $M\times I$ with respect to the time variable. The above  formula makes it straightforward to check that such an integration commutes with the Lie algebroid differentials, namely that we have:
\begin{equation}\label{eq:differentialcommuteswithint}
\int_0^1 (\d_{A_M\times I} \alpha)_t dt= \d_{A_M} \int_0^1 \alpha_t dt.
\end{equation}

\begin{lemma}\label{lemm:timedependentliealgmorphisms}
There is a one-to-one correspondence between smooth families of vector bundle maps $\Phi_t:A_M \to A_N$ and morphisms of graded algebras:
$$\Phi^*:(\Omega(A_N), {}\wedge)\to (\Omega(A_M)^I, {}\wedge).$$

Furthermore, $\Phi_t:A_M \to A_N$ defines a smooth family of Lie algebroid morphisms if and only if the induced application
 $\Phi^*:(\Omega(A_N), {}\wedge,\d_{A_N})\to (\Omega(A_M)^I, {}\wedge,\d_{A_M\times I})$ 
 is a morphism of DGAs.
\end{lemma}
\begin{proof}
The first statement essentially follows from the definitions: a \emph{smooth family} of vector bundle maps $\Phi_t:A_M \to A_N$ is, by definition, a vector bundle map  $\Phi:A_M\times I \to A_N$, which is well-known  to be the same as a morphism of graded algebras $\Phi^*:\Omega(A_N)\to \Omega(A_M\times I)=:\Omega(A_M)^I$.

Once smoothness is established, the second statement easily follows from \eqref{prop:timedependentproduct} and \eqref{prop:timedependentdifferential}.
\end{proof}

The Proposition \ref{prop:derivationandtheta} then has its time-dependent counterpart, whose proof is obtained by simply replacing $M$ with $M\times I$, and following the definitions.

\begin{prop}\label{prop:derivationandthetatimedependent}
Let  $\Phi_t:A_M\to A_N$ be a smooth family of vector bundle maps covering $\phi_t:M\to N$.  We denote by $\Phi^*:\Omega(A_N)\to \Omega(A_M)^I$ the induced map on forms.

There is a one-to-one correspondence between:
\begin{itemize}
\item degree $-1$ $\Phi^*\!$-derivations
\item smooth families of time-dependent sections $\theta_t\in\Gamma(\phi_t^*A_N)$ supported by $\phi_t$.
\end{itemize} 
More precisely, to any section supported by $\phi$, $\theta\in\Gamma(\phi^*A_N)$ one associates the $\Phi^ *$-derivation $\cV:=i_{\theta}^{\Phi}$ defined by:
\begin{equation}\label{eq:leibnizopreatordeftimedependent}
\bigl\langle i_{\theta}^{\Phi}\beta\, ,\, a_1{}_\wedge\dots{}_\wedge a_{k-1} \bigr\rangle
:=
\bigl\langle \beta\circ \phi\, ,\, \theta{}_\wedge \Phi\circ  a_1{}_\wedge\dots {}_\wedge  \Phi\circ a_{k-1}\bigr\rangle\in C^\infty(M\times I)
\end{equation}
where $\beta\in \Omega^k(A_N)$ and $a_1,\dots ,a_{k-1}\in \Gamma(A_M)$. Furthermore, any $\Phi^*\!$-derivation is of this form.
\end{prop}

\bibliographystyle{habbrv}

\end{document}